\documentclass[11pt]{amsart}

\usepackage{amsthm,amsmath,amssymb,amscd,amsfonts,latexsym,xcolor,hyperref}

\theoremstyle{plain}
\newtheorem{theorem}{Theorem}[section]

\newtheorem{proposition}[theorem]{Proposition}
\newtheorem{fundamental fact}[theorem]{Fundamental Fact}

\newtheorem{corollary}[theorem]{Corollary}

\newtheorem{def-thm}[theorem]{Definition-Theorem}
\newtheorem{lemma}[theorem]{Lemma}

\theoremstyle{definition}
\newtheorem{definition}[theorem]{Definition}

\newtheorem{remark}[theorem]{Remark}
\newtheorem{example}[theorem]{Example}

\newtheorem*{acknowledgement}{Acknowledgement}

\newcommand{\CC}{\mathbb{C}}

\DeclareMathOperator{\Ric}{Ric}

\DeclareMathOperator{\Ker}{Ker}
\newcommand{\rminus}{{r^-}}

\newcommand{\R}{{\mathbb R}}

\newcommand{\C}{{\mathbb C}}

\DeclareMathOperator{\kod}{kod}

\allowdisplaybreaks

\begin{document}

\title[On the zero set of holomorphic sectional curvature]{On the zero set of holomorphic sectional curvature}

\author{Yongchang Chen}
\address{Yongchang Chen. Department of Mathematics, University of Houston, 4800 Calhoun Road, Houston, TX 77204, USA} \email{{ychen224@central.uh.edu}}

\author{Gordon Heier}
\address{Gordon Heier. Department of Mathematics, University of Houston, 4800 Calhoun Road, Houston, TX 77204, USA} \email{{heier@math.uh.edu}}

\subjclass[2020]{32Q05, 32Q10, 32Q15, 53C55}
\keywords{Complex manifolds, K\"ahler metrics, semi-definite holomorphic sectional curvature}

\thanks{The second author was supported by a grant from the Simons Foundation (Grant Number 963755-GH)}

\begin{abstract}A notable example due to Heier, Lu, Wong, and Zheng shows that there exist compact complex K\"ahler manifolds with ample canonical line bundle such that the holomorphic sectional curvature is negative semi-definite and vanishes along high-dimensional linear subspaces in every tangent space. The main result of this note is an upper bound for the dimensions of these subspaces. Due to the holomorphic sectional curvature being a real-valued bihomogeneous polynomial of bidegree $(2,2)$ on every tangent space, the proof is based on making a connection with the work of D'Angelo on complex subvarieties of real algebraic varieties and the decomposition of polynomials into differences of squares. Our bound involves an invariant that we call the holomorphic sectional curvature square decomposition length, and our arguments work as long as the holomorphic sectional curvature is semi-definite, be it negative or positive.
\end{abstract}

\maketitle

\section{Introduction}
\label{chap:intro}
\subsection{Background information on earlier work and statements of the main theorems}
On a complex K\"ahler manifold, the relationship between Ricci curvature and holomorphic sectional curvature has long been considered to be somewhat mysterious \cite[Remark 7.20]{Zhengbook}. Several decades ago, S.-T. Yau conjectured that a (projective) K\"ahler manifold with negative holomorphic sectional curvature possesses a (different) K\"ahler metric of negative Ricci curvature, i.e., has ample canonical line bundle. Yau also conjectured that a (projective) K\"ahler manifold with positive holomorphic sectional curvature is rationally connected and even unirational. Not much progress on these types of questions was made until a little over a decade ago, papers by Heier-Lu-Wong \cite{heier_lu_wong_mrl} in the negative curvature case and Heier-Wong \cite{heier_wong_cag} in the positive curvature case revived the topic. For more on the recent history of this area, we refer the reader to the literature (e.g., \cite{heier_lu_wong_mrl}, \cite{heier_wong_cag},  \cite{HLW_JDG}, \cite{Wu_Yau_Invent},\cite{Wu_Yau_CAG}, \cite{Tosatti_Yang}, \cite{Yang}, \cite{HLWZ}, \cite{Diverio_Trapani}, \cite{YZ}, \cite{heier_wong_doc_math}), as we are only focussing on one particular question pertaining to the possible existence of zeros of the holomorphic sectional curvature in this note.\par
During their investigation of the consequences of semi-negative holomorphic sectional curvature, Heier-Lu-Wong \cite[Remark 1.6]{HLW_JDG} raised the question if it is possible for a K\"ahler manifold with ample canonical line bundle to carry a K\"ahler metric whose semi-negative holomorphic sectional curvature has ``many zeros" at each point. Subsequently, Heier-Lu-Wong-Zheng \cite[Example 1.2]{HLWZ} gave the following example which showed that it is possible for the semi-negative holomorphic sectional curvature on a compact K\"ahler manifold with ample canonical line bundle to vanish on large linear subspaces of the tangent spaces at every point. Before we present this example, we give a precise definition of the invariant $r^-$ that captures the size of these subspaces. Note that analogous non-compact examples in the semi-positive holomorphic sectional curvature case where recently constructed in \cite{chen_m_heier_arxiv}.

\begin{definition}
	Let $M$ be a Hermitian manifold with semi-negative holomorphic sectional curvature $H$. For $p\in M$, let $\eta(p)$ be the maximum of those integers $k\in \{0,\ldots,n:=\dim M\}$ such that there exists a $k$-dimensional subspace $L\subset T_p M$ with $H(v)=0$ for all $v\in L\backslash \{0\}$. Set $\eta_M:=\min_{p\in M} \eta(p)$ and $\rminus:=n-\eta_M$. Note that by definition $\rminus=0$ if and only if $H$ vanishes identically. Also, $\rminus=n$ if and only if there exists at least one point $p\in M$ such that $H$ is strictly negative at $p$. Moreover, $\eta(p)$ is upper-semicontinuous as a function of $p$, and consequently the set
	$$\{p\in M\ |\ n-\eta(p)=\rminus\}=\{p\in M \mid \eta(p)=\eta_M\} $$
	is an open set in $M$ (in the classical topology).
\end{definition}\par

\begin{example}{\cite[Example 1.2]{HLWZ}} \label{theta_div_example}Let $A$ be a principally polarized abelian variety of dimension $n+1$, and let $M\subset A$ be a theta divisor. It was proven by Andreotti-Mayer that for generic $A$, the hypersurface $M$ is non-singular, i.e., a  submanifold. Take the standard flat K\"ahler metric on $A$ and restrict it to $M$. By the curvature decreasing property of subbundles, $M$ has semi-negative bisectional curvature and, in particular, $H\leq 0$. By the adjunction formula, the canonical line bundle of $M$ is ample, i.e., $c_1(M) < 0$. On the other hand, the following explicit computation shows that $\rminus=\lfloor\frac{n+1}{2}\rfloor$, which is less than $n$ (when $n$ is at least $2$).\par
	Let $p\in M$ be a fixed point. Let $U$ be a neighborhood of $p$ in $A$ with coordinates $z_0, z_1,\ldots,z_n$ such that $p$ is the origin and $T_pM=\left(\frac{\partial}{\partial z_0}\right)^\perp$. Then $z_1,\dots,z_n$ can be used as local holomorphic coordinates on $M\cap U$ and $M$ defined as a graph by $z_0=f(z_1,\ldots,z_n)$ for some holomorphic function $f$ with $f(0,\ldots,0)=0$ and $df(0,\ldots,0)=0$. The induced metric on $M$ is the graph metric which is given in components as
	$$g_{i\overline{j}} = \delta_{i j} +f_i\overline{f_j},$$
	where $f_i = \frac{\partial f}{\partial z_i}$. From this, we obtain that the components of the curvature tensor at the point $p$ are
	\begin{equation}\label{theta_div_curv_tensor} R_{i\overline{j} k \overline{l}}=-f_{ik}\overline{f_{jl}},\end{equation}
	where $f_{ij} = \frac{\partial^2 f}{\partial z_i\partial z_j}$. Therefore, for a unit vector $v=
	\sum_{i=1}^n v_i \frac{\partial}{\partial z_i} \in T_pM$, the holomorphic sectional curvature of $v$ is
	$$R_{v\overline{v}v\overline{v}}=-|f_{vv}|^2,$$
	where $f_{vv} = \sum_{i,j=1}^n v_iv_j f_{ij}$. Thus, for any non-zero vector $v\in T_pM$, $H(v)=0$ if and only if $f_{vv}=0$. Note that all such $v$ form a quadratic cone $Q$ in $T_p M=\CC^n$. According to \cite[page 735, Proposition]{GH}, every quadratic cone of dimension $m$ contains a linear space of dimension $\lfloor\frac{m-1}{2}\rfloor+1$, which in our situation yields $\rminus\leq \lfloor\frac{n+1}{2}\rfloor$ (note that $m=n-1$). Moreover, if the quadratic cone is non-degenerate, there exist no linear spaces of larger dimension on it. We will conclude the discussion of this example by showing that there exists a point $p$ of $M$ where $Q$ is indeed non-degenerate and can therefore infer that $\rminus= \lfloor\frac{n+1}{2}\rfloor$.\par
	Let us assume that the matrix $(f_{ij})$ is degenerate at $p$. Then there exists a non-zero tangent vector $v\in T_pM$ such that $f_{iv}:=\sum_{j=1}^n v_j f_{ij}=0$ for all $i=1,\ldots,n$. Consequently, by \eqref{theta_div_curv_tensor}, $R_{i\overline{j}v\overline{v}}=0$ for all $i,j=1,\ldots,n$. In particular, the $n$-fold wedge of the Ricci curvature form with itself vanishes at $p$. Now, if this holds for all $p\in M$, then $c_1(M)^n =0$, contradicting the initial observation that $c_1(M)<0$. \hfill $\qed$
\end{example}

It would seem likely that the discrepancy between the positivity or negativity of the canonical line bundle and the amount of zeros of the semi-definite holomorphic sectional curvature cannot exceed a certain size. In the extreme case that the holomorphic sectional curvature on a projective K\"ahler manifold $(M,g)$ vanishes identically, it is known that the entire curvature tensor of $g$ vanishes and $M$ is an abelian variety up to a finite unramified covering, see the comments in the proof of the more general Proposition \ref{incompatibility_thm}.\par

In this note, we offer the following main theorem in this regard. In our theorem's assumptions, we only require the first Chern class of the canonical line bundle to be definite (either positive or negative), so we formulate it in terms of the following invariant $r_0$:

\begin{definition}\label{defr0}
	Let $M$ be a Hermitian manifold with holomorphic sectional curvature $H$. For $p\in M$, let $\eta_0(p)$ be the maximum of those integers $k\in \{0,\ldots,n:=\dim M\}$ such that there exists a $k$-dimensional subspace $L\subset T_x M$ with $H(v)=0$ for all $v\in L\backslash \{0\}$. Set $\eta_0:=\min_{p\in M} \eta_0(p)$ and $r_0:=n-\eta_0$. 
\end{definition}

\begin{theorem}\label{main 1}
Let $M$ be an $n$-dimensional compact K\"ahler manifold whose holomorphic sectional curvature is either positive or negative semi-definite. Assume that the first real Chern class of the canonical line bundle of $M$ is definite with the opposite sign as the holomorphic sectional curvature and that $M$ has holomorphic sectional curvature square decomposition length $N$. Then we have 
$$r_0\ge n-\lfloor\frac{N}{N+1}n\rfloor.$$
\end{theorem}

For the definition of {\it holomorphic sectional curvature square decomposition length}, see Definition \ref{sqdl_def}. The symbols $\lfloor \cdot \rfloor$ denote the round down of a rational number.\par

The general philosophy expressed here is that the discrepancy between the nature of the canonical line bundle and the zero values of the holomorphic sectional curvature cannot be too big. Therefore, the above main theorem can be refined to the following statements in terms of the numerical invariants defined in Section \ref{pos_notions}. First, without an assumption of projectivity, we formulate a theorem in terms of $n_R(M)$, measuring the kernel dimensions of the curvature tensor $R$.
\begin{theorem}\label{main 2}
	For an $n$-dimensional compact K\"ahler manifold $M$ whose holomorphic sectional curvature is either positive or negative semi-definite, we have 
$$r_0\ge n-\lfloor\frac{Nn+(n-n_{R}(M))}{N+1}\rfloor.$$
\end{theorem}
In the case of the holomorphic sectional curvature being semi-negative on a projective K\"ahler manifold, based on the inequality \eqref{invar_ineq}, we immediately obtain the following corollary in terms of the numerical dimension $\nu(M)$. Note that it was proven in \cite[Theorem 1.4]{HLW_JDG} that the canonical line bundle of a projective K\"ahler manifold with semi-negative holomorphic sectional curvature is nef.
\begin{corollary}\label{cor_main 2}
	For an $n$-dimensional projective K\"ahler manifold whose holomorphic sectional curvature is semi-negative, we have $$r_0\ge n-\lfloor\frac{Nn+(n-\nu(M))}{N+1}\rfloor.$$
\end{corollary}

\subsection{Introductory comments on the (semi-)positivity of polynomials}
The new direction pursued in this note is to investigate holomorphic sectional curvature via the properties it enjoys by virtue of being a real-valued bihomogeneous polynomial on every tangent space. It therefore seems worthwhile to include a few brief comments about (semi-)positive (and thereby also (semi-)negative) polynomials in general.

For polynomials, (semi-)positivity conditions and being a sum of squares are closely related, although they are not quite equivalent in general and the precise nature of the problem is quite delicate.\par

In the case of real polynomials in one variable, we do have an equivalence. Namely,  for $f \in \R[x]$, the inequality $f(x)\ge 0$ holds for all $x \in \R$ if and only if  $f(x)=p^2(x)+q^2(x)$, for some $p(x),q(x) \in \R[x]$. The proof is well-known and is left as an exercise.\par

However, such an equivalence fails in the case of a real polynomial in several variables, i.e., an element of $\R[x_1,\dots,x_n]$. In fact, a well known counterexample is given by the so-called Motzkin polynomial $f(x, y)=x^4y^2+x^2y^4-3x^2y^2+1$. Although $f(x, y)\ge 0$ for all $(x, y) \in \R^2$, it cannot be written as sum of any number of squares (see \cite{Motzkin}).\par

A precise characterization of the semi-positivity for elements of $\R[x_1,\dots,x_n]$ is given by Hilbert's seventeenth problem:

For $f \in \R[x_1,\dots, x_n]$, the inequality $f(x_1,\dots, x_n)\ge 0$ holds for all $(x_1,\dots, x_n) \in \R^n$ if and only if $$f(x_1,\dots, x_n)=\frac{p_1^2(x_1,\dots, x_n)+\dots+p_N^2(x_1,\dots, x_n)}{q_1^2(x_1,\dots, x_n)+\dots+q_N^2(x_1,\dots, x_n)},$$ for some positive integer $N$ and $p_i(x_1,\dots, x_n),q_i(x_1,\dots, x_n) \in \R[x_1,\dots, x_n]$. Originally posed as an open problem by Hilbert in his 1900 list of problems, this statement is a theorem due to Emil Artin. Later works by other authors address questions of effectivity in finding this decomposition. We refer the reader to the literature for more on this. \par

Moving on to the realm of complex numbers, for real-valued bihomogeneous polynomials $f(z, \bar z)$, there is a characterization for strict positivity due to D'Angelo \cite[Theorem 1]{D'Angelo 00} as follows. The inequality $f(z, \bar z)>0$ holds for all $z\ne 0\in \C^n$ if and only if there is a non-negative integer $d$ and a positive definite Hermitian matrix $(E_{\mu\nu})$, such that 
$$f(z,\bar z)=\frac{\sum E_{\mu\nu}z^\mu\bar z^\nu}{\|z\|^{2d}}.$$

On the other hand, for semi-positive real-valued bihomogeneous polynomials, it is impossible in general to write them as quotients of sums of squares. The reference \cite[pp.~171--172]{D'Angelo 00} gives two examples, namely $(|z_1|^2-|z_2|^2)^2\ge 0$ and $(|z_1z_2|^2-|z_3|^4)^2+|z_1|^8\ge 0$ for all $(z_1,z_2)\in \C^2$, and shows they cannot be written as a quotient of sums of squares as they violate the following necessary condition for being a quotient of sums of squares:

For $f(z, \bar z)=\frac{||p(z)||^2}{||q(z)||^2}$, consider the germ of a holomorphic mapping $t\mapsto z(t)$ from $\C$ to $\C^n$ such that the order of vanishing of the numerator is at least the order of vanishing of the denominator. Then the pullback function $t\mapsto p(z(t), \overline{z(t)})$ is either identically 0, or vanishes to finite even order and its initial form is independent of the argument of $t$.

Consequently, in the consideration of semi-definite holomorphic sectional curvature, the notion of being a sum of squares (or even quotient of sum of squares) is not general enough. We were therefore naturally led to the discussion in \cite{D'Angelobook} where decomposition methods into differences of squares are discussed, and it is this method of decomposition into differences that we apply in this note to the study of semi-definite holomorphic sectional curvature.

\section{Definitions, Notations and Preliminary Material}
\label{chap:defsandprel}

\subsection{Curvature}
A complex manifold will be denoted by $M$, its complex dimension by $n$. For a local chart of $M$, $z=(z_1,\dots,z_n)$ will be used to denote its coordinates. The corresponding bases for the holomorphic and antiholomorphic tangent spaces will be denoted by $\frac{\partial}{\partial z_1},\ldots,\frac{\partial}{\partial z_n}$ and $\frac{\partial}{\partial\bar z_1},\ldots, \frac{\partial}{\partial\bar z_n}$, respectively. The symbols $dz_1,\ldots,dz_n$ and $d\bar z_1,\ldots,d\bar z_n$ denote the dual bases for the holomorphic and antiholomorphic cotangent spaces.

The canonical line bundle of $M$, which is the top wedge of the holomorphic cotangent bundle, will be denoted by $K$. The first Chern class of $M$ is defined to be $c_1(-K)$ and will be denoted by $c_1(M)$.

A Hermitian metric on $M$ will be denoted by $g$. In local coordinates it is of the form
\begin{equation*}
	g=\sum_{i,j} g_{i\bar j} dz_i\otimes d\bar{z}_j.
\end{equation*}
Here the summation runs from $1$ to $n$ tacitly and $(g_{i\bar j})$ is Hermitian. It is customary to denote by 
$$\omega=\frac{\sqrt{-1}}{2}\sum_{i,j} g_{i\bar{j}} dz_i\wedge d\bar{z}_{j}$$ 
the {\it (1,1)-form associated to $g$}. The metric $g$ is said to be {\it K\"ahler} if the exterior derivative $d\omega = 0$.

For a Hermitian metric, there is a unique distinguished connection on the holomorphic tangent bundle, known as the {\it Chern connection}, associated to it. The curvature associated to this connection will be denote by $R$. In local coordinates, the curvature $4$-tensor $R_{i\bar jk\bar l}$ is given by the formula:
\begin{equation*}
	R_{i\bar j k \bar l}=-\frac{\partial^2 g_{i\bar j}}{\partial z_k\partial \bar z _l}+\sum_{p,q=1}^n g^{p\bar q}\frac{\partial g_{i\bar p}}{\partial z_k}\frac{\partial g_{q\bar j}}{\partial \bar z_l}.
\end{equation*}
Here $g^{p\bar q}$ refers to the inverse matrix of $(g_{i\bar j})$ (see \cite[p.\ 1104, \S 1]{Wu}). To distinguish it from other types of curvature derived from it, we will also call $R$ the {\it full curvature}.

From this formula, we immediately have a symmetry property for the curvature involving conjugation, namely
\begin{equation}\label{symm}
	\overline{R_{j\bar il\bar k}}=R_{i\bar jk\bar l}.
\end{equation}
For a K\"ahler metric, its curvature possesses the further symmetry properties
\begin{equation} \label{Ksymm}
	R_{i\bar jk\bar l}=R_{k\bar ji\bar l}=R_{i\bar lk\bar j}.
\end{equation}
If $g$ is K\"ahler, there is only one notion of Ricci curvature (for the definitions of the first, second and third Ricci curvatures in the Hermitian case, see \cite[p.~181]{Zhengbook}), which amounts to defining the \textit{Ricci
	curvature form} to be
\begin{equation*}
	Ric=-\sqrt{-1}\partial\bar\partial \log\det(g_{i\bar{j}}).
\end{equation*}
Moreover, if the tangent vector $v=\sum_{i=1}^n v_i e_i$ and curvature tensor are written in terms of a unitary frame $e_1,\ldots,e_n$, then 
\begin{equation*}
	Ric(v):=Ric(v,\bar v)=\sum_{i,j,k=1}^n R_{i\bar j k \bar k}v_i\bar v_j.
\end{equation*}
By a result of Chern, the class of the form $\frac{1}{2\pi}Ric$ is equal to $c_1(M)$.\par
The {\it scalar curvature} $S$ is defined to be the trace of the Ricci curvature with respect to the metric $g$, i.e., the function
\begin{equation*}
	\sum_{i=1}^n Ric(e_i)=\sum_{i,j=1}^n R_{i\bar i j \bar j}.
\end{equation*}
 It follows from linear algebra and the
definition of scalar curvature that
$$
Ric \wedge \omega^{n-1} = \tfrac {2}{ n} S\, \omega^n.
$$

\par

If $v=\sum_{i=1}^n v_i \frac{\partial }{\partial z_i}$ is a non-zero tangent vector at $p\in M$, then the {\it holomorphic sectional curvature} $H(v)$ is given by the real number
\begin{equation}\label{HSC}
	H(v)=\left(\sum_{i,j,k,l=1}^n R_{i\bar j k \bar l}(p)v_i\bar v_jv_k\bar v_l\right) / \left(\sum_{i,j,k,l=1}^ng_{i\bar j}g_{k\bar l} v_i\bar v_jv_k\bar v_l\right).
\end{equation}
An important fact about holomorphic sectional curvature is the following. If $M'$ is a submanifold of $M$, then the holomorphic sectional curvature of $M'$ does not exceed that of $M$. To be precise, if $v$ is a non-zero tangent vector to $M'$, then
\begin{equation*}
	H'(v)\leq H(v),
\end{equation*}
where $H'$ is the holomorphic sectional curvature associated to the metric on $M'$ induced by $g$. This property was key in the construction of Example \ref{theta_div_example}. For a short direct proof of this inequality see \cite[Lemma 1]{Wu}. Basically, the inequality is an immediate consequence of the Gauss-Codazzi equation.\par
We have the following pointwise result due to Berger \cite{Berger} (see \cite{Hall_Murphy} for a recent new approach).
\begin{theorem}[\cite{Berger}]\label{berger_theorem}
	Let $M$ be a compact manifold with a K\"ahler metric of semi-negative holomorphic sectional curvature. Then the scalar curvature function 
	$S$ is also semi-negative everywhere on $M$. Moreover, let $p\in M$ and assume that there exists $w\in T_pM\backslash \{\vec{0}\}$ such that $H(w) < 0$. Then $S(p) < 0$.
\end{theorem}
Berger's theorem is proven using a pointwise formula expressing the scalar curvature at a point in terms of the average holomorphic sectional curvature on the unit sphere in the tangent space at that point. Based on Berger's theorem, we have the following proposition from \cite[Proposition 2.2]{HLW_JDG}. We reproduce the short proof of the proposition for the convenience of the reader and to provide context for the work in this note.
\begin{proposition}\label{incompatibility_thm}
	Let $M$ be a projective manifold whose first real Chern class is zero. Let $g$ be a K\"ahler metric on $M$ whose holomorphic sectional curvature is semi-negative. Then the holomorphic sectional curvature of $g$ vanishes identically, and $M$ is an abelian variety up to a finite unramified covering.
\end{proposition}
\begin{proof}
	Assume the holomorphic sectional curvature of $g$ does not vanish identically. Then there exists a point $p\in M$ and $w\in T_p M\backslash \{\vec{0}\}$ such that $H(w) < 0$. By Theorem \ref{berger_theorem}, the scalar curvature is non-positive everywhere, and $S(p) < 0$. Thus,
	\begin{align}
		0 & =2\pi \int _M c_1(-K_M) \wedge \omega^{n-1}\nonumber\\
		&  =\int _M \Ric (g)\wedge \omega^{n-1}\nonumber\\
		& =\int_M  \frac 2 n S\, \omega ^n <0 \nonumber,
	\end{align}
	which is a contradiction.\par
	Having shown that the holomorphic sectional curvature of $g$ does vanish identically, it is immediate that $M$ is an abelian variety up to a finite unramified covering. Namely, it is a basic fact that the holomorphic sectional curvature of a K\"ahler metric completely determines the curvature tensor $R$ (\cite[Proposition 7.1, p. 166]{kobayashi_nomizu_ii}). In particular, if $H$ vanishes identically, then $R$ vanishes identically. However, due to \cite{Igusa}, a projective K\"ahler manifold with vanishing curvature tensor admits a finite unramified covering by an abelian variety.
\end{proof}
\subsection{Numerical measures of positivity}\label{pos_notions}
First of all, for the statement of Theorem \ref{main 2} we introduce a measure of degeneracy that does not require the canonical line bundle to be nef, namely the (minimal) dimension of the kernel subspaces of the curvature tensor on $M$. 

\begin{definition}\label{R_ker_def}
	Let $M$ be a K\"ahler manifold of dimension $n$. For $p\in M$, let $L_p$ be the set of all $v\in T_p M$ with $R_{v\bar j k \bar l}=0$ for all $j,k,l$ in some local coordinate system. The set $L_p$  is a linear subspace of $T_p M$ and we set $n_R (M) = n-\min\ \{\dim L_p \mid p\in M\}$.
\end{definition}
\begin{remark}
	Due to the symmetry properties enjoyed by the curvature of a K\"ahler metric, the vector $v$ can be placed in any of the 4 possible positions without affecting the definition.
\end{remark}
Furthermore, we introduce three standard notions of positivity for line bundles from algebraic geometry. Let $L$ be an arbitrary nef line bundle on the projective manifold $M$. Then the {\it numerical dimension of $L$}, which we denote $\nu(L),$ is $\max\{k\in\{0,1,\ldots,\dim M\}: (c_1(L))^k\not =  [0]\}$, where $c_1(L)$ denotes the first real Chern class of $L$. In the case $L=K$, we write $\nu(M)$ for $\nu(K)$, the so-called {\it numerical dimension of $M$} (aka the {\it numerical Kodaira dimension of $M$}). 
\begin{remark}\label{num_dim_zero_rkm}
	It is immediate that $\nu(L)=0$, i.e., $c_1(L)=[0]$, implies that $L$ is numerically trivial. The converse also holds true, which is nicely explained in \cite[Remark 1.1.20]{Laz_I}.\par
\end{remark}
The notion of nef dimension is based on the following theorem (\cite{Tsuji}, \cite{8aut}).
\begin{theorem}
	Let $L$ be a nef line bundle on a normal projective variety $M$. Then there exists an almost holomorphic dominant rational map $f:M\dashrightarrow Y$ with connected fibers, called a ``reduction map,'' such that
	\begin{enumerate}
		\item $L$ is numerically trivial on all compact fibers $F$ of $f$ with $\dim F = \dim M - \dim Y$, and
		\item for every general point $x\in M$ and every irreducible curve $C$ passing through $x$ with $\dim f(C) > 0$, we have $L.C > 0$.
	\end{enumerate}
	The map $f$ is unique up to birational equivalence of $Y$.
\end{theorem}
We call $\dim Y$ the {\it nef dimension of $L$}. When we apply the above theorem with $L=K_M$, we call $n(M):=\dim Y$ the {\it nef dimension of $M$}. 

Finally, the following notion is based not just on numerical information, but on actual sections of $L$.
Let $L$ be a line bundle on $M$. The {\it Kodaira-Iitaka dimension} $\kod(L)$ of $L$ is $-\infty$ if $H^0(M,kL)=0$ for all positive integers $k$. Otherwise, it is the degree of $H^0(M,kL)$ as a polynomial in $k$. If $L=K$, then we call $\kod(K)=:\kod(M)$ the Kodaira dimension of $M$. Note that $\kod(L)\in \{-\infty,0,1,2,\ldots,\dim M\}$.

Recall that it follows from the Riemann-Roch theorem and the Kodaira vanishing theorem that for an ample line bundle $L$, we have $\kod(L)=\dim M$.

On a projective manifold $M$ with nef canonical line bundle $K_M$, the following chain of inequalities holds:
\begin{equation} \kod(M)\leq \nu(M) \leq n(M).\label{chain_of_inequ}\end{equation}
The first inequality was established by Kawamata in \cite[Proposition 2.2]{Kawamata_85_pluri_sys} and the second inequality is in \cite[Proposition 2.8]{8aut}. The name Abundance Conjecture is commonly used to refer to the claim that the first inequality is actually an equality, i.e., that Kodaira dimension and numerical dimension of $M$ agree (\cite[Conjecture 7.2]{Kawamata_85_pluri_sys}). \par
Furthermore, it is known that the Abundance Conjecture actually implies $\kod(M)= n(M)$, making \eqref{chain_of_inequ} an all-around equality. For more on this, see \cite[Section 4]{HLW_JDG}.\par
Finally, in the case of a projective K\"ahler manifold with semi-negative holomorphic sectional curvature, we can deduce from \cite[Theorem 1.4]{HLW_JDG} and \cite[Theorems 1.5 and 1.7]{HLWZ} that
\begin{equation}
	n(M) \geq n_R (M) \geq \nu(M). \label{invar_ineq}
\end{equation}
as in this case the property of being {\it truly flat} (\cite[Definition 1.3]{HLWZ}) is equivalent to the property in Definition \ref{R_ker_def}.
\section{Decomposition of the Holomorphic Sectional Curvature into Differences of Squares}
\label{Decomposition}

\subsection{D'Angelo's decomposition into differences of squares}
The main idea behind Theorems \ref{main 1} and \ref{main 2} is that on the tangent space at every point, the holomorphic sectional curvature is a real-valued bihomogeneous complex polynomial of bidegree $(2,2)$. Such a polynomial can be decomposed into a difference of squares as described in \cite{D'Angelo 82}. In that paper, D'Angelo's intention was to study subelliptic estimates for the $\bar \partial$-Neumann problem of a boundary of a finite type domain in $\C^n$ and the relation with orders of contact of complex analytic varieties. We now provide a quick summary of D'Angelo's method, mainly following the monograph \cite{D'Angelobook}.\par
For the purpose of this discussion, we take $z_1,\ldots,z_n$ to be complex variables on $\CC^n$ which we use to form polynomials
\begin{equation*}
r(z,\bar z) = \sum_{\alpha,\beta}c_{\alpha\beta}z^\alpha \bar z^\beta
\end{equation*}
with complex coefficients $c_{\alpha\beta}$. Note that this represents a minor conflict of notation with the earlier stipulation that $z=(z_1,\dots,z_n)$ is used for the coordinates of a local chart of a complex manifold $M$. We make this temporary choice to stay in line at this point with the notation in D'Angelo's work. We hope no significant confusion arises as a result. We also temporarily use the letter $M$ to denote a real algebraic variety in $\CC^n$, in contrast with the use of the same letter for a K\"ahler manifold in Theorems \ref{main 1} and \ref{main 2} and other places. Furthermore, the letter $g$ has different meanings in different places: it denotes the K\"ahler metric but is also used for holomorphic functions in the decompositions into differences of squares.\par

It is elementary to see that $r(z,\bar z)$ is real-valued on $\CC^n$ if and only if the matrix consisting of the coefficients is Hermitian. Furthermore, it is easy to see that $r(z,\bar z)$ vanishes identically on $\CC^n$ if and only if all the coefficients $c_{\alpha\beta}$ are equal to zero. We can use real-valued polynomials to define real algebraic varieties as follows:
\begin{definition}
A {\it real algebraic variety} is a subset $M$ of $\C^n$ which is the zero set of a real-valued polynomial. In other words, $M=\{z\in\C^n \mid r(z,\bar z)=0\}$, where $r$ is a real-valued polynomial on $\C^n$.
\end{definition}
An important special case of a real-valued polynomial is the following:
\begin{definition}
We say that the real-valued polynomial $r$ is {\it bihomogeneous of bidegree $(m,m)$} if 
\begin{equation*}
r(z,\bar z) = \sum_{|\alpha|=|\beta|=m}c_{\alpha\beta}z^\alpha \bar z^\beta.
\end{equation*}
\end{definition}
By a polarization argument, one can prove the following decomposition theorem.
\begin{theorem}
For a real-valued polynomial $r(z,\bar z)$ that is bihomogeneous of bidegree at least $(1,1)$, there exist two sets of homogeneous polynomials $f_1,\dots,f_N$ and $g_1,\dots,g_N$ in $\CC[z_1,\ldots,z_n]$, all having vanishing constant terms, such that
	\begin{equation*}
		r(z,\bar z)=(|f_1(z)|^2+\dots +|f_N(z)|^2)-(|g_1(z)|^2+\dots +|g_N(z)|^2).
	\end{equation*}
\end{theorem}
Based on this observation, an important statement can be made about irreducible complex algebraic or analytic varieties contained in real algebraic varieties.
\begin{theorem}\label{superset}
	Let $M$ be a real algebraic variety given by $r(z,\bar z)=(|f_1(z)|^2+\dots +|f_N(z)|^2)-(|g_1(z)|^2+\dots +|g_N(z)|^2)$, which is bihomogeneous of bidegree at least $(1,1)$. Then the family of complex algebraic varieties $V(f-Ug)$, as $U$ varies over all $N\times N$ unitary matrices, contains a superset for each irreducible complex variety contained in $M$ (depending on the complex variety).
\end{theorem}
Note that we wrote $f$ for the vector of functions $(f_1,\ldots,f_N)$ and $g$ for the vector of functions $(g_1,\ldots,g_N)$. It is easy to check that $V(f-Ug)\subset M$ for all $N\times N$ unitary matrices $U$.\par

From the defining formula of the holomorphic sectional curvature \eqref{HSC}, combined with symmetry property \eqref{symm} of the full curvature, we have:
\begin{theorem}
	Let $p\in M$ and $T_pM \cong \C^n$ the holomorphic tangent space at $p$. Working in variables $(z_1,\ldots,z_n)$ on the tangent space $T_pM \cong \C^n$, the numerator part of the holomorphic sectional curvature 
	$$\sum_{i,j,k,l} R_{i\bar j k \bar l}z_i\bar z_jz_k\bar z_l=\sum_{i,j,k,l} R_{i\bar j k \bar l}z_iz_k\bar z_j\bar z_l$$ is a real-valued bihomogenous polynomial of bidegree $(2,2)$.
\end{theorem}
Note that in the defining formula of holomorphic sectional curvature, it is the numerator that carries most of the information. In particular, the holomorphic sectional curvature is zero if and only if the numerator of the expression \eqref{HSC} is zero.

\subsection{Further auxiliary results in the K\"ahler case}
In several places, it is crucial for our arguments that we work with a K\"ahler metric. We therefore establish further auxiliary results pertaining to the K\"ahler case that we will later exploit. \par
It is well-known (see \cite[Lemma 7.19]{Zhengbook}) that the holomorphic sectional curvature of a K\"ahler metric completely determines the full curvature $R$. The next Lemma considers the case when the holomorphic sectional curvature has a special form, resulting in a corresponding special form of the full curvature. This observation will be used in the proofs of Theorems \ref{main 1} and \ref{main 2}.

\begin{lemma} \label{recovery} Let $x\in M$. If the holomorphic sectional curvature takes the form 
	$$H(v)=\sum_{p=1}^N|\sum_{i, k}f_{ik}^pv_iv_k|^2-\sum_{p=1}^N|\sum_{i, k}g_{ik}^pv_iv_k|^2$$
	with fixed complex numbers $f_{ik}^p, g_{ik}^p$ for all unit holomorphic tangent vectors $v=\sum_{i=1}^n v_i \frac{\partial }{\partial z_i}\in T_xM$, then the full curvature is of the form
	$$R_{i\bar j k \bar l}=\sum_{p=1}^N f_{ik}^p\overline{f_{jl}^p}-\sum_{p=1}^N g_{ik}^p\overline{g_{jl}^p}.$$
\end{lemma}
\begin{proof}
	Due to the fact that holomorphic sectional curvature completely determines the full curvature, we can verify the result by showing that the holomorphic sectional curvature of the above full curvature equals the original holomorphic sectional curvature:
	\begin{align*}
		H(v) & = \sum_{i,j,k,l} R_{i\bar j k \bar l}v_i\bar v_j v _k\bar v_l\\
		& =\sum_{i,j,k,l}(\sum_{p=1}^N f_{ik}^p\overline{f_{jl}^p}-\sum_{p=1}^N g_{ik}^p\overline{g_{jl}^p})v_i\bar v_j v _k\bar v_l\\
		& =\sum_{i,j,k,l}\sum_{p=1}^N f_{ik}^p\overline{f_{jl}^p}v_i\bar v_j v _k\bar v_l-\sum_{i,j,k,l}\sum_{p=1}^N g_{ik}^p\overline{g_{jl}^p}v_i\bar v_j v _k\bar v_l\\
		& =\sum_{p=1}^N\sum_{i,j,k,l} f_{ik}^pv_iv _k\overline{f_{jl}^p}\bar v_j\bar v_l-\sum_{p=1}^N\sum_{i,j,k,l} g_{ik}^pv_iv _k\overline{g_{jl}^p}\bar v_j\bar v_l\\
		& =\sum_{p=1}^N|\sum_{i, k}f_{ik}^pv_iv_k|^2-\sum_{p=1}^N|\sum_{i, k}g_{ik}^pv_iv_k|^2.
	\end{align*}
	So the lemma follows.
\end{proof}
The next two lemmas can also be thought of as (partially) recovering the curvature tensor. They generalize Lemma 2.1 of \cite{HLWZ}, and will be needed for the proofs of Theorems \ref{main 1} and \ref{main 2}.
\begin{lemma}\label{lemma 1} Let $M$ be a K\"ahler manifold and let $p\in M$, $v \in T_pM$. If $R_{v\bar w w \bar w}=0$ for all $w\in T_pM$, then $R_{v\bar x y \bar z}=0$ for all $x, y, z\in T_pM$.
\end{lemma}
\begin{proof}
For $w=(w_1,\dots ,w_n)$, we expand the expression $R_{v\bar w w \bar w}$:
\begin{align*}
& R_{v\bar w w \bar w} \\
=\ & \sum_{i, j, k}R_{v\bar ij\bar k}\bar w_iw_j\bar w_k\\
=\ & \sum_{\alpha, \beta}R_{v\bar \alpha \beta\bar \alpha}\bar w_\alpha w_\beta \bar w_\alpha+\sum_{\alpha<\gamma, \beta}(R_{v\bar \alpha \beta\bar \gamma}+R_{v\bar \gamma \beta\bar \alpha})\bar w_\alpha w_\beta\bar w_\gamma.
\end{align*}	
This is a polynomial in $w, \bar w$ vanishing identically. Therefore, all the coefficients in this polynomial are equal to zero, namely
\begin{align*}
R_{v\bar \alpha \beta\bar \alpha}& =0\quad \forall  \alpha, \beta\\ 
(R_{v\bar \alpha \beta\bar \gamma}+R_{v\bar \gamma \beta\bar \alpha}) & =0  \quad \forall  \alpha, \beta, \gamma \text{ with }\alpha<\gamma.
\end{align*}
By the K\"ahler symmetry properties \eqref{Ksymm}, we have  
\begin{equation*}
R_{v\bar \alpha \beta\bar \gamma}+R_{v\bar \gamma \beta\bar \alpha} = 2 R_{v\bar \alpha \beta\bar \gamma} =2 R_{v\bar \gamma \beta\bar \alpha}=0,
\end{equation*}
proving the lemma.
\end{proof}

\begin{lemma}\label{up}
Let $M$ be a K\"ahler manifold with semi-definite (either positive or negative) holomorphic sectional curvature. Let $p\in M$, $v \in T_pM$. If $R_{x\bar y  v \bar v}=0$ for all $x, y\in T_pM$, then $R_{x\bar y z \bar v}=0$ for all $x, y, z\in T_pM$.
\end{lemma}
\begin{proof}
	For any fixed $w$, we consider the quantity $R_{\lambda v+w, \overline{\lambda v+w}, \lambda v+w, \overline{\lambda v+w}}$. This is a real-valued polynomial in the complex variable $\lambda$ and its conjugate $\bar\lambda$. Using the assumption, we expand it to find:
	\begin{equation*}
		R_{\lambda v+w, \overline{\lambda v+w}, \lambda v+w, \overline{\lambda v+w}}=2Re(\lambda^2R_{v\bar wv\bar w})+2Re(\lambda R_{v\bar w w\bar w})+R_{w\bar ww\bar w}.
	\end{equation*}
	Specifically, note that the fourth order term vanishes, since it equals 
$$\lambda\bar \lambda\lambda\bar \lambda R_{v\bar vv\bar v}.$$ 
and the assumption applies. The third order term vanishes due to the assumption and the symmetry properties \eqref{symm} and \eqref{Ksymm}. Due to the semi-definiteness assumption, both the quadratic and the linear terms of $\lambda$ in the above expression must also vanish, which yields
	\begin{equation*}
		R_{v\bar wv\bar w}=0=R_{v\bar w w\bar w}.
	\end{equation*}
	Furthermore, again by the symmetry properties \eqref{symm} and \eqref{Ksymm}, we have
	\begin{equation*}
		R_{v\bar ww\bar w}=\overline{R_{w\bar ww\bar v}}=0
	\end{equation*}
	and thus
	\begin{equation*}
		R_{w\bar ww\bar v}=0.
	\end{equation*}
	The result follows by Lemma \ref{lemma 1}.
\end{proof}
We close this chapter by introducing the holomorphic sectional curvature square decomposition length.
\begin{definition} \label{sqdl_def}
	For a K\"ahler manifold $M$, its {\it holomorphic sectional curvature square decomposition length} $N$ is defined to be: $N=\max_{x\in M} \min_{\Gamma}$\hfill \{there exists a chart at $x$ such that the holomorphic sectional curvature at $x$ has the form $$H(v)=\sum_{p=1}^\Gamma |\sum_{i, k}f_{ik}^pv_iv_k|^2-\sum_{p=1}^\Gamma|\sum_{i, k}g_{ik}^pv_iv_k|^2$$ 
for any unit holomorphic tangent vector $v=\sum_{i=1}^n v_i \frac{\partial }{\partial z_i}\in T_xM$\}.
\end{definition}

\section{Linear and Quadratic Algebra}
\label{Linandquadalg}

This chapter collects the key linear and quadratic algebra facts that are needed in proving Theorems \ref{main 1} and \ref{main 2}. All vector spaces are assumed to be finite dimensional (complex) vector spaces.
\subsection{Some auxiliary results from linear algebra}

For the linear algebra part, first recall the standard inclusion-exclusion principle for the intersection of two linear subspaces.
\begin{theorem}\label{linear lemma}
	For two linear subspaces $V_1$ and $V_2$ inside a vector space $V,$ the following equality holds:
	\begin{equation*}\label{i-e}
		\dim(V_1\cap V_2)=\dim V_1+\dim V_2-\dim (V_1+ V_2).
	\end{equation*}
\end{theorem}
For intersections of more than two linear subspaces, we estimate from below the dimension of their intersection as follows.
\begin{theorem}\label{intersection linear}
	For $N$ linear subspaces $V_1,\dots, V_N$ inside a vector space $V,$ the dimension of their intersection satisfies the inequality
	\begin{equation*}
		\dim (V_1\cap\ldots \cap V_N)\ge\dim V_1+\ldots +\dim V_N -(N-1)\dim V.
	\end{equation*}
\end{theorem}
\begin{proof}
For the base case $N=2$, we have
\begin{align*}
& \dim(V_1\cap V_2)\\
=\ & \dim V_1+\dim V_2-\dim (V_1+ V_2)\\
\ge\ & \dim V_1+\dim V_2-(2-1)\dim V.
\end{align*} 
	Here, the equality comes from Thereom \ref{i-e}, while the inequality comes from the obvious fact $\dim(V_1+V_2)\le\dim V$.
	
	Then we do induction on $N$:
	\begin{align*}
		& \dim (V_1\cap\ldots \cap V_{N+1})\\
		=\ & \dim ((V_1\cap\ldots \cap V_N)\cap V_{N+1})\\
		\ge\ &\dim (V_1\cap\ldots \cap V_N)+\dim V_{N+1}-\dim V\\
		\ge\ & \dim V_1+\ldots +\dim V_N -(N-1)\dim V+\dim V_{N+1}-\dim V\\
		=\ &\dim V_1+\ldots +V_{N+1}-N\dim V.
	\end{align*}
	The first inequality is by the base case $N=2$, and the second equality is by the induction assumption. This gives the result.
\end{proof}
\begin{remark}
	Considering the coordinate hyperplanes in $\C^n$ shows that this inequality is sharp.
\end{remark}
From this inequality, we deduce a condition to guarantee the non-zeroness of the intersection of linear subspaces which will be used in proving Theorem \ref{main 1}.
\begin{corollary}\label{nonzero intersection linear}
	If $N$ linear subspaces $V_1,\dots, V_N$ inside a vector  space $V$  satisfy the dimension inequality $\dim V_1+\ldots +\dim V_N> (N-1)\dim V,$ then they have non-zero intersection.
\end{corollary}
\begin{proof}
	We have
	\begin{equation*}
		\dim (V_1\cap\ldots \cap V_N)\ge\dim V_1+\ldots +\dim V_N -(N-1)\dim V>0
	\end{equation*}
	where the first inequality comes from Theorem \ref{intersection linear} and the second inequality comes from the assumption. So the result follows.
\end{proof}
More generally, we have the following condition to guarantee a prescribed minimum dimension of intersection of linear subspaces, which will be used in proving Theorem \ref{main 2}.
\begin{corollary}\label{nonzero intersection linear 2}
	If $N$ linear subspaces $V_1,\dots, V_N$ inside a vector space $V$  satisfy the dimension inequality $\dim V_1+\ldots +\dim V_N> (N-1)\dim V+k$, then they have at least a $(k+1)$-dimensional linear subspace in common.
\end{corollary}
\begin{proof}
	This time the inequality becomes:
	\begin{equation*}
		\dim (V_1\cap\ldots \cap V_N)\ge\dim V_1+\ldots +\dim V_N -(N-1)\dim V>k,
	\end{equation*}
	where the change is in the second inequality, which comes from the new assumption. The result follows again.
\end{proof}
\subsection{Some auxiliary results about quadrics}
For the quadratic algebra aspect, we are interested in maximal linear subspaces contained in quadrics.
\begin{theorem}\label{max dim linear quadric}
	For a quadric of rank $r$ in $\C^n$, the dimension of any linear subspace contained in the quadric is no greater than  $(n-r)+\lfloor\frac{r}{2}\rfloor$.
\end{theorem}
\begin{proof}
	We first give a proof of the theorem for the case of a full rank quadric, i.e., $r=n$, then pass to the general quadric by forming a quotient. For the full rank case, a proof can be found in \cite[p. 735]{GH}. However, we wish to provide a different and elementary proof here which is more in line with the methods used in this note. Namely, we assume there exists a linear subspace of dimension $\lfloor\frac{n}{2}\rfloor+1$ in order to derive a contradiction. Take a basis of this subspace and complement it to a basis of $\CC^n$. In the new basis, the matrix of the quadric will have the form
	\[\begin{pmatrix}
		0 & *\\
		* & *
	\end{pmatrix},\]
	where the square submatrix pertaining to the first $\lfloor\frac{n}{2}\rfloor+1$ rows and columns is $0$. However, such a matrix is easily seen to have determinant equal to zero, contradicting the assumption of full rank.\par
	For arbitrary rank $r$, we take the quadric's kernel space $K\subset \C^n$, which has dimension $n-r$, and form the corresponding quadric in the quotient space $\C^n/K$, which now becomes a full rank quadric. Applying the proof of the full rank case to images of linear subspaces under the projection map $\C^n\to \C^n/K$ finishes the proof. \end{proof}
From Theorem \ref{intersection linear} together with Theorem \ref{max dim linear quadric}, we deduce a condition to guarantee non-zero intersection of kernels of several quadrics, which will be used in proving Theorem \ref{main 1}.
\begin{corollary}\label{intersection ker}
	Let $Q_1,\dots ,Q_N$ be quadrics in $\C^n$. If they share a linear subspace of $\dim=\lfloor\frac{Nn}{N+1}\rfloor+1$, then their kernels have nonzero intersection.
\end{corollary}
\begin{proof}
	Denote the common linear subspace by $L$. For each quadric $Q_i$, $L$ is contained in one of its maximal linear subspaces, say $L_i$. Thus we have two containment relations:
	\begin{equation}\label{L contained1}
		L\subseteq L_i,
	\end{equation}
	\begin{equation}\label{Ker contained1}
		\Ker Q_i\subseteq L_i.
	\end{equation}
	Note that the containment \eqref{Ker contained1} comes from the maximality condition on $L_i$. The containment 	\eqref{L contained1} combined with Theorem \ref{max dim linear quadric} yields
	\begin{equation*}
		\dim L\le(n-r_i)+\lfloor\frac{r_i}{2}\rfloor.
	\end{equation*}
	Removing the round down $\lfloor\cdot \rfloor$ then leads to a bound for $r_i$:
	\begin{equation}\label{bound r1}
		r_i\le 2(n-\dim L).
	\end{equation}
	On the other hand, applying Theorem \ref{linear lemma} to \eqref{L contained1} and \eqref{Ker contained1} yields
	\begin{equation*}
		\dim(\Ker Q_i\cap L)\ge \dim L-\lfloor\frac{r_i}{2}\rfloor.
	\end{equation*}
	Summing all the inequalities from $1$ to $N$ yields
	\begin{equation*}
		\dim(\Ker Q_1\cap L)+\dots+\dim(\Ker Q_N\cap L)\ge N\dim L-(\lfloor\frac{r_1}{2}\rfloor+\dots+\lfloor\frac{r_N}{2}\rfloor).
	\end{equation*}
	Removing the round down $\lfloor\cdot \rfloor$ and combining with \eqref{bound r1} yields
	\begin{equation}\label{interstep}
		\dim(\Ker Q_1\cap L)+\dots+\dim(\Ker Q_N\cap L)\ge 2N \dim L-Nn.
	\end{equation}
	Furthermore, combining with the assumption $\dim L=\lfloor\frac{Nn}{N+1}\rfloor+1$ yields
	\begin{equation*}
		\dim(\Ker Q_1\cap L)+\dots+\dim(\Ker Q_N\cap L)>(N-1)\dim L.
	\end{equation*}
	Now Corollary \ref{nonzero intersection linear} gives the result.
\end{proof}
\begin{remark} The statement of Corollary \ref{intersection ker} is sharp with respect to the dimension of the shared linear subspace. This can be seen from the following example.
\end{remark}

\begin{example}\label{sharp_ex} To simplify the notation, we let $\eta=\lfloor\frac{Nn}{N+1}\rfloor$, and define the following quadrics in terms of coordinates $z_1,\ldots,z_n$ on $\CC^n$.
	\begin{align*}
		Q_1\ =\ &\{z_1z_{\eta+1}+z_{2}z_{\eta+2}+\dots+z_{n-\eta}z_n=0\}\\
		Q_2\ =\ &\{z_{n-\eta+1}z_{\eta+1}+\dots+z_{2(n-\eta)}z_n=0\}\\
		\vdots\ & \\
		Q_{\lfloor\frac{\eta}{n-\eta}\rfloor}\ =\ &\{z_{(\lfloor\frac{\eta}{n-\eta}\rfloor-1)(n-\eta)+1}z_{\eta+1}+\dots+z_{\eta}z_{2\eta-(\lfloor\frac{\eta}{n-\eta}\rfloor-1)(n-\eta)}=0\}\\
		Q_{\lfloor\frac{\eta}{n-\eta}\rfloor+1}\ =\ & \C^n\\
		\vdots\ & \\
		Q_N\ =\ &  \C^n.
	\end{align*}
	Note that for the purpose of this example, we trivially consider $\C^n$ to be the quadric given by the zero matrix and that this construction is valid due to
	\begin{equation*}
		\lfloor\frac{Nn}{N+1}\rfloor\le N(n-\lfloor\frac{Nn}{N+1}\rfloor).
	\end{equation*}
	Furthermore, all these quadrics share the common linear subspace
	\begin{equation*}
		L=\{(z_1,\dots,z_n)\in \C^n|z_{\eta+1}=0,\dots,z_n=0\}.
	\end{equation*}
	On the other hand, their kernels are
	\begin{align*}
		\Ker(Q_1)\ = \ &\left\{(z_1,\dots,z_n)\in \C^n| z_1=0,\dots,z_{n-\eta}=0,\right. \\
		&\left. z_{\eta+1}=0,\dots, z_n=0\right\}\\
		\Ker(Q_2)\ = \ & \left\{(z_1,\dots,z_n)\in \C^n|z_{n-\eta+1}=0,\dots,z_{2(n-\eta)}=0,\right. \\
		&\left. z_{\eta+1}=0,\dots, z_n=0\right\}\\
		\vdots\ &\\	
		\Ker(Q_{\lfloor\frac{\eta}{n-\eta}\rfloor})\ = \ & \left \{(z_1,\dots,z_n)\in\C^n|(z_{\lfloor\frac{\eta}{n-\eta}\rfloor-1)(n-\eta)+1}=0,\dots,z_{\eta}=0,\right.\\
		&\left.z_{\eta+1}=0,\dots,z_{2\eta-(\lfloor\frac{\eta}{n-\eta}\rfloor-1)(n-\eta)}=0\right\}\\
\Ker( Q_{\lfloor\frac{\eta}{n-\eta}\rfloor+1})\ = \ & \CC^n \\
\vdots\ &\\	
\Ker( Q_N)\ = \ & \CC^n 		
\end{align*}
	and have zero intersection according to the way the indexing is arranged.
\end{example}

More generally, we have the following condition to guarantee a prescribed minimum dimension for the intersection of kernels of several quadrics, which will be used in proving Theorem \ref{main 2}. This condition can be seen to be sharp with an example analogous to Example \ref{sharp_ex}.
\begin{corollary}\label{intersection ker 2}
	Let $Q_1,\dots ,Q_N$ be quadrics in $\C^n$. If they share a linear subspace of dimension equal to $\lfloor\frac{Nn+k}{N+1}\rfloor+1$, then their kernels have a $(k+1)$-dimensional linear subspace in common.
\end{corollary}
\begin{proof}
	This time, the assumption becomes $\dim L=\lfloor\frac{Nn+k}{N+1}\rfloor+1$, and combining it with the inequality \eqref{interstep} leads us to
	\begin{equation*}
		\dim(\Ker Q_1\cap L)+\dots+\dim(\Ker Q_N\cap L)>(N-1)\dim L+k.
	\end{equation*}
	Then Corollary \ref{nonzero intersection linear 2} yields the result.
\end{proof}

\section{Proof of Theorem \ref{main 1}}
\label{mthm_proof}

We now proceed to the proof of Theorem \ref{main 1}.
\begin{proof}
	By the assumption on the square decomposition length of the holomorphic sectional curvature, at any point $x$ in the K\"ahler manifold $M$, we can find a chart such that the numerator of the holomorphic sectional curvature has the expression:
	$$H(v)=\sum_{p=1}^N|\sum_{i, k}f_{ik}^pv_iv_k|^2-\sum_{p=1}^N|\sum_{i, k}g_{ik}^pv_iv_k|^2 $$
	for a tangent vector $v=\sum_{i=1}^n v_i \frac{\partial }{\partial z_i}$ at $x$.
	Therefore, the zero set of $H$ in $T_xM$ is 
	$$\{(v_1,\dots,v_n)\in \CC^n \mid \sum_{p=1}^N\mid \sum_{i, k}f_{ik}^pv_iv_k|^2-\sum_{p=1}^N|\sum_{i, k}g_{ik}^pv_iv_k|^2=0\}.$$
	Now, we proceed by contradiction. 
	
	Assume that the statement of Theorem \ref{main 1} is false. Then by the definition of $r_0$ as in Definition \ref{defr0}, this means at every point $x\in M$, there is a linear subspace of dimension $n-(n-\lfloor\frac{N}{N+1}n\rfloor)+1=\lfloor\frac{Nn}{N+1}\rfloor+1$, such that the holomorphic sectional curvature is zero along this subspace. Take such a linear subspace of dimension $\lfloor\frac{Nn}{N+1}\rfloor+1$, and denote it by $L$.
	
	By Theorem \ref{superset}, this complex linear subspace $L$, as a complex variety contained in a real algebraic variety defined by a bihomogenous bidegree $(2,2)$ polynomial $H$, must be contained in $V(f-Ug)$ for some $N\times N$ unitary matrix $U=(u_j^p)$. Write this out explicitly:
	\begin{equation*}
		L\subseteq \{(v_1,\dots,v_n)\in\C^n|\sum_{i, k}f^p_{ik}v_iv_k-\sum_{j=1}^N\sum_{i, k}u^p_jg^j_{ik}v_iv_k=0, \forall p=1,\dots, N \}.
	\end{equation*}
	After a change of the order of the summation, the condition here is equivalent to
	\begin{equation}\label{Nquadrics}
		\sum_{i, k}(f^p_{ik}-\sum _{j=1}^Nu^p_jg^j_{ik})v_iv_k=0\qquad\text{for all } p=1,\dots\,N.
	\end{equation}
	These $N$ quadrics in $\C^n=T_xM$ share the linear subspace $L$ of dimension equal to $\lfloor\frac{Nn}{N+1}\rfloor+1$. By Corollary \ref{intersection ker}, there is a non-zero vector $v$ in the  intersection of their kernels. In other words, there exists
	\begin{equation}\label{v}
		v=(v_1,\dots,v_n)\ne 0  \text{ such that } \sum_{k}(f^p_{ik}-\sum_{j=1}^Nu^p_jg^j_{ik})v_k=0 \text{ for all } i \text{ and } p.
	\end{equation}
	On the other hand, from Lemma \ref{recovery}, we can recover the full curvature to be
	$$R_{i\bar j k \bar l}=\sum_{p=1}^N f_{ik}^p\overline{f_{jl}^p}-\sum_{p=1}^N g_{ik}^p\overline{g_{jl}^p}.$$
	Next, we verify that for all $i, j$, the common kernel element $v$ of those $N$ quadrics will give us a zero of $R_{i\bar j v \bar v}$:
	\begin{align*}
		R_{i\bar j v \bar v} & = \sum_{k, l}(\sum_{p=1}^N f_{ik}^p\overline{f_{jl}^p}-\sum_{p=1}^N g_{ik}^p\overline{g_{jl}^p})v_kv_l\\
		& =\sum_{p=1}^N\sum_{k, l}f^p_{ik}v_k\overline{f^p_{jl}}\overline{v_l}-\sum_{p=1}^N\sum_{k, l}g^p_{ik}v_k\overline{g^p_{jl}}\overline{v_l}\\
		&=\sum_{p=1}^N\sum_{k, l, m, n}u^p_mg^m_{ik}v_k\overline{u^p_n}\overline{g^n_{jl}}\overline{v_l}-\sum_{p=1}^N\sum_{k, l}g^p_{ik}v_k\overline{g^p_{jl}}\overline{v_l}\\
		&=0.
	\end{align*}
	The second equality is substitution by \eqref{v}, and the last equality follows from the fact that $U$ is a unitary matrix.
	
	So for this non-zero $v$ we have $R_{i\bar j v \bar v}=0$ for all $i, j$. By Lemma \ref{up} this implies
	\begin{equation*}
		R_{i\bar j k \bar v}=0\quad\text{for all}\ i, j, k.
	\end{equation*}
	Taking the trace with respect to the first two indices we get
	\begin{equation*}
		\sum_{i,j} g^{i\bar j} R_{i\bar j k \bar v}=0\quad\text{for all}\ k.
	\end{equation*}
	Then it follows that the determinant
	\begin{equation*}
		\det(\sum_{i,j} g^{i\bar j} R_{i\bar j k\bar l})=0,
	\end{equation*}
	which is equivalent to $Ric^n(x)=0$. Since this happens at all points $x\in M$, we obtain $c_1^n(M)=0$. However, this is a contradiction, as we assumed the first Chern class of $M$ to be either positive or negative definite.
\end{proof}

\section{Proof of Theorem \ref{main 2}}
\label{lower_rank_case}

The proof of Theorem \ref{main 2} differs from that of Theorem \ref{main 1} only by making a slightly more general argument in the handling of the kernel vectors. The essence of the proof remains the same.

\begin{proof}
	We begin by noting that in case $n_R=0$, the lower bound in the Theorem equals
	$$n-\lfloor\frac{Nn+(n-n_{R}(M))}{N+1}\rfloor=n-\lfloor\frac{Nn+n}{N+1}\rfloor=n-\lfloor\frac{n(N+1)}{N+1}\rfloor=n-n=0$$
	and the theorem is vacuous. In fact, if $n_R=0$, then the entire curvature tensor and all curvatures derived from it vanish identically. In particular, $r_0=0$.\par
	In case $n_R>0$, we again proceed by contradiction and assume that the statement of Theorem \ref{main 2} is false. Then by the definition of $r_0$, this means at every point $x\in M$, there is a linear subspace of dimension $n-(n-\lfloor\frac{Nn+(n-n_{R})}{N+1}\rfloor)+1=\lfloor\frac{Nn+(n-n_{R})}{N+1}\rfloor+1$, such that the holomorphic sectional curvature is zero along this subspace. \par
	The arguments in the proof of Theorem \ref{main 1} and an application of Corollary \ref{intersection ker 2} show that for every $x\in M$, there is a subspace $L$ of dimension at least $(n-n_R)+1$ in $T_x M$ such that for all $v \in L$, we have $R_{i\bar j v \bar v}=0$ for all $i, j$, and thus by Lemma \ref{up}
	\begin{equation*}
		R_{i\bar j k \bar v}=0\quad\text{for all}\quad i, j, k.
	\end{equation*}
	This is a contradiction to the definition of $n_R$.	
\end{proof}

\section{Sharpness}
\label{Sharpness}

\subsection{The local point of view}\label{local_view}
From the local point of view, the bound obtained in Theorem \ref{main 1} (and also in Theorem \ref{main 2}) can readily seen to be sharp. In order to do so, we produce here a full curvature tensor locally such that the bound of \ref{main 1} is attained. And since we are working from a local point of view, after a local change of coordinate if necessary, we can assume from the beginning that we are working in normal coordinates, i.e., $g_{i\bar j}=\delta_{i j}$ at the origin.\par

Also, since the full curvature and the holomorphic sectional curvature determine each other, we produce a holomorphic sectional curvature $H$ locally first, and then recover from it the full curvature. The definition of this $H$ is based on Example \ref{sharp_ex}. Again, to simplify the notation, we set $\eta=\lfloor\frac{Nn}{N+1}\rfloor$. In coordinates $z_1,\ldots,z_n$ on $\CC^n$, let
\begin{align*}
	H=&\left|z_1z_{\eta+1}+z_{2}z_{\eta+2}+\dots+z_{n-\eta}z_n\right|^2+\dots+\\
	&\left|z_{(\lfloor\frac{\eta}{n-\eta}\rfloor-1)(n-\eta)+1}z_{\eta+1}+\dots+z_{\eta}z_{2\eta-(\lfloor\frac{\eta}{n-\eta}\rfloor-1)(n-\eta)}\right|^2.
\end{align*}
Note that this $H$ is in the form of a sum of squares (as opposed to the more general  difference of squares), so it is clearly semi-positive definite. As in Example \ref{sharp_ex}, the zero set of $H$ is
\begin{equation*}
	\{z\in\C^n \mid  z_{\eta+1}=0,\dots,z_n=0\}.
\end{equation*}
This is a linear subspace in $\C^{n}$ of dimension $n-\lfloor\frac{Nn}{N+1}\rfloor$, which is the bound in Theorem \ref{main 1}. On the other hand, when we rewrite $H$ into
\begin{multline*}
	H=\left(z_1z_{\eta+1}+\dots+z_{n-\eta}z_n\right)\left(
	\bar z_1\bar z_{\eta+1}+\dots+\bar z_{n-\eta}\bar z_n\right)+\ldots+\\
	\left(z_{(\lfloor\frac{\eta}{n-\eta}\rfloor-1)(n-\eta)+1}z_{\eta+1}+\dots+z_{\eta}z_{2\eta-(\lfloor\frac{\eta}{n-\eta}\rfloor-1)(n-\eta)}\right)\\
	\left(\bar z_{(\lfloor\frac{\eta}{n-\eta}\rfloor-1)(n-\eta)+1}\bar z_{\eta+1}+\dots+\bar z_{\eta}\bar z_{2\eta-(\lfloor\frac{\eta}{n-\eta}\rfloor-1)(n-\eta)}\right)
\end{multline*}
and multiply everything out, we can observe: for $i\ne j$, terms of the form $z_iz_k\bar z_j\bar z_k$ do not appear in $H$, so from the recovery formula in Lemma \ref{recovery}, when $i\ne j$, $R_{i\bar jk\bar k}=0$. On the other hand, for all $i$, there is exactly one index $k$ such that $z_iz_k\bar z_i\bar z_k$ appears (with a coefficient of $1$), so $R_{i\bar ik\bar k}=1$, and for $l\not = k$, $z_iz_l \bar z_i\bar z_l$ does not appear.\par

Consequently, for the Ricci curvature we have
\[Ric(e_i,e_j)=\sum_{k, l} g^{k\bar l}R_{i\bar jk\bar l}=\sum_k R_{i\bar jk\bar k}=\left\{ \begin{array}{rl}1,~\text{for } i=j\\
	0,~\text{for } i\ne j.
\end{array} \right.\]
The second equality comes from the normal coordinate, the final equality comes from the above observation for the full curvature.
So the $Ric$ matrix is just the identity matrix at the origin, and, perhaps somewhat surprisingly given the many zeros of $H$, it is positive definite. This gives the desired example.

\subsection{The global point of view}
In order to provide an actual example of a compact manifold that establishes the sharpness of Theorem \ref{main 1} (and Theorem \ref{main 2}), it seems tempting to generalize the construction from Example \ref{theta_div_example} to higher codimension. In Example \ref{theta_div_example}, for a point $p\in M$ and a neighborhood $U$ of $p$ in $A$ with coordinates $z_0, z_1,\ldots,z_n$ such that $p$ is the origin and $T_pM=\left(\frac{\partial}{\partial z_0}\right)^\perp$, we wrote $M$ as a graph defined by $z_0=f(z_1,\ldots,z_n)$ for some holomorphic function $f$ with $f(0,\ldots,0)=0$ and $df(0,\ldots,0)=0$. The induced metric on $M$ is the graph metric which is given in components as
$$g_{i\overline{j}} = \delta_{i j} +f_i\overline{f_j},$$
where $f_i = \frac{\partial f}{\partial z_i}$. From this, we obtained that the components of the curvature tensor at the point $p$ are
\begin{equation*}R_{i\overline{j} k \overline{l}}=-f_{ik}\overline{f_{jl}}.
\end{equation*}\par
The generalized situation	discussed in this note is the one where the local representation as a graph is of the form
$$w_1= f^1(z_1,\ldots,z_n), \ldots,w_c=f^c(z_1,\ldots,z_n),$$
and the induced graph metric is given in components as
$$g_{i\overline{j}} = \delta_{i j} +\sum_{s=1}^c f^s_i\overline{f^s_j},$$
with the curvature tensor at the point $p$ being
\begin{equation*}R_{i\overline{j} k \overline{l}}=\sum_{s=1}^c -f^s_{ik}\overline{f^s_{jl}}.
\end{equation*}
It appears to be unclear however how such a high codimension embedding can be chosen so that the holomorphic functions $f^1,\ldots,f^c$ have exactly the highly specific behavior exhibited in Section \ref{local_view}. \par
In other words, it seems tempting based on this note to conjecture that there is no bound in Theorem \ref{main 1} that only depends on the dimension of $n$, and in particular does not depend on the holomorphic sectional curvature square decomposition length, other than the obvious bound $r_0 \geq 1$. A high codimensional example as described above would constitute a proof of this conjecture. We hope to return to this question in the future.

\begin{acknowledgement}
This article is based on the first author's Ph.D. dissertation, written under the direction of the second author at the University of Houston.
\end{acknowledgement}

\end{document}